\documentclass[11pt,a4paper,reqno]{amsart}
\usepackage[english]{babel}
\usepackage[T1]{fontenc}
\usepackage{verbatim}
\usepackage{palatino}
\usepackage{amsmath}
\usepackage{amssymb}
\usepackage{amsthm}
\usepackage{amsfonts}
\usepackage{graphicx}
\usepackage{esint}
\usepackage{stmaryrd}
\usepackage{color}
\usepackage{mathtools}

\usepackage[colorlinks = true, citecolor = black]{hyperref}
\title{Metric currents and the Poincar\'e inequality}
\author{Katrin F\"assler and Tuomas Orponen}
\thanks{{K.F. was supported by Swiss National Science Foundation via the project \emph{Intrinsic rectifiability and mapping theory on the Heisenberg group}, grant no. $161299$.} T.O. was supported by the Academy of Finland via the Research Fellowship project \emph{Quantitative rectifiability in Euclidean and non-Euclidean spaces}, grant no. 309365.}
\address{Department of Mathematics\\ University of Fribourg \\ Chemin du Mus\'{e}e 23, CH-1700 Fribourg, Switzerland}
\email{katrin.faessler@unifr.ch}
\address{Department of Mathematics and Statistics, University of Helsinki, Gustaf H\"allstr\"ominkatu 2b, 00014 University of Helsinki, Finland}
\email{tuomas.orponen@helsinki.fi}

\newcommand{\R}{\mathbb{R}}

\newcommand{\N}{\mathbb{N}}
\newcommand{\Q}{\mathbb{Q}}

\newcommand{\Z}{\mathbb{Z}}

\newcommand{\calV}{\mathcal{V}}

\newcommand{\calD}{\mathcal{D}}
\newcommand{\calH}{\mathcal{H}}

\newcommand{\calN}{\mathcal{N}}
\newcommand{\calB}{\mathcal{B}}

\newcommand{\calS}{\mathcal{S}}

\newcommand{\calU}{\mathcal{U}}

\newcommand{\spt}{\operatorname{spt}}

\newcommand{\diam}{\operatorname{diam}}
\newcommand{\card}{\operatorname{card}}
\newcommand{\dist}{\operatorname{dist}}

\newcommand{\Lip}{\textup{Lip}}

\def\Barint_#1{\mathchoice
          {\mathop{\vrule width 6pt height 3 pt depth -2.5pt
                  \kern -8pt \intop}\nolimits_{#1}}%
          {\mathop{\vrule width 5pt height 3 pt depth -2.6pt
                  \kern -6pt \intop}\nolimits_{#1}}%
          {\mathop{\vrule width 5pt height 3 pt depth -2.6pt
                  \kern -6pt \intop}\nolimits_{#1}}%
          {\mathop{\vrule width 5pt height 3 pt depth -2.6pt
                  \kern -6pt \intop}\nolimits_{#1}}}

\numberwithin{equation}{section}

\theoremstyle{plain}
\newtheorem{thm}[equation]{Theorem}

\newtheorem{lemma}[equation]{Lemma}

\newtheorem{ex}[equation]{Example}
\newtheorem{cor}[equation]{Corollary}

\theoremstyle{definition}

\newtheorem{definition}[equation]{Definition}

\theoremstyle{remark}
\newtheorem{remark}[equation]{Remark}

\begin{document}

\begin{abstract} We show that a complete doubling metric space $(X,d,\mu)$ supports a weak $1$-Poincar\'e inequality if and only if it admits a \emph{pencil of curves} (PC) joining any pair of points $s,t \in X$. This notion was introduced by S. Semmes in the 90's, and has been previously known to be a sufficient condition for the weak $1$-Poincar\'e inequality.

Our argument passes through the intermediate notion of a \emph{generalised pencil of curves} (GPC). A GPC joining $s$ and $t$ is a normal $1$-current $T$, in the sense of Ambrosio and Kirchheim, with boundary $\partial T = \delta_{t} - \delta_{s}$, support contained in a ball of radius $\sim d(s,t)$ around $\{s,t\}$, and satisfying $\|T\| \ll \mu$, with
\begin{displaymath} \frac{d\|T\|}{d\mu}(y) \lesssim \frac{d(s,y)}{\mu(B(s,d(s,y)))} + \frac{d(t,y)}{\mu(B(y,d(t,y)))}. \end{displaymath}
We show that the $1$-Poincar\'e inequality implies the existence of GPCs joining any pair of points in $X$. Then, we deduce the existence of PCs from a recent decomposition result for normal $1$-currents due to Paolini and Stepanov.
\end{abstract}
\maketitle

\section{Introduction}

Let $(X,d,\mu)$ be a complete metric space, where $\mu$ is a doubling locally finite Borel measure. It is known, see for example \cite{MR1683160, MR3363168}, that plenty of analysis can be conducted on $(X,d,\mu)$ whenever the \emph{weak $p$-Poincar\'e inequality}
\begin{equation}\label{PPoincare} \fint_{B} |u - u_{B}| \, d\mu \leq C \diam(B) \left( \fint_{\lambda B} \rho^{p} \, d\mu \right)^{1/p} \end{equation}
is satisfied for some $C,p,\lambda \geq 1$, for all balls $B\subset X$, for all locally integrable Borel functions $u \colon X \to \R$, and for all upper gradients $\rho$ of $u$. So, it is worthwhile to find necessary and sufficient conditions for the validity of \eqref{PPoincare}. One well-known sufficient condition is the existence of \emph{pencils of curves}, introduced by Semmes \cite{MR1414889} in the $90$'s. To motivate the results in the present paper, we first discuss Semmes' condition in some detail; our definition is the one given in Section 14.2 in \cite{MR3363168}, where the setting is somewhat more general than in Semmes' original work \cite{MR1414889}.
\begin{definition}[Pencils of curves]\label{PC} The space $(X,d,\mu)$ admits \emph{pencils of curves} (PC) if there exists a constant $C_0 \geq 1$ with the following property. For all distinct $s,t\in X$ there is a family $\Gamma_{s,t}$ of rectifiable curves $\gamma \subset B(s,C_0 d(s,t))$, each joining $s$ to $t$ and satisfying $\calH^{1}(\gamma) \leq C_{0}d(s,t)$, and a probability measure $\alpha_{s,t}$ on $\Gamma_{s,t}$ such that
\begin{displaymath}
\int_{\Gamma_{s,t}}\int_{\gamma} g \; d\calH^{1} \, d\alpha_{s,t}(\gamma) \leq C_0 \int_{B(s,C_0 d(s,t))} \frac{g(y)}{\Theta(s,d(s,y))} + \frac{g(y)}{\Theta(t,d(t,y))} \, d\mu(y).
\end{displaymath}
for Borel functions $g: X \to [0,\infty]$. Here, and in the sequel, $\Theta$ stands for the $1$-dimensional density
\begin{displaymath} \Theta(x,r) = \frac{\mu(B(x,r))}{r}, \qquad x \in X, \: r > 0. \end{displaymath}
\end{definition}

A doubling space $(X,d,\mu)$ admitting PCs satisfies the weak $1$-Poincar\'e inequality. This was proven by Semmes for $Q$-regular spaces, see \cite[Theorem B.15] {MR1414889}, and the general case can be found for instance in Heinonen's book \cite[Chapter 4]{MR1800917}.

Of course, Semmes in \cite{MR1414889} also gives sufficient conditions for finding PCs: his Standard Assumptions (see \cite[Theorem 1.11]{MR1414889} and above) include the space $(X,d,\mu)$ to be an orientable topological $n$-manifold, with $\mu = \calH^{n}$. Moreover, $X$ has to be locally contractible (for more precise statements, see \cite[Definition 1.7]{MR1414889} or \cite[Definition 1.15]{MR1414889}, but also the discussion in \cite[Remark A.35]{MR1414889}). These assumptions are certainly not necessary for a space $(X,d,\mu)$ to admit PCs or support a Poincar\'e inequality; notably, the Laakso spaces \cite{MR1748917} have PCs, hence satisfy \eqref{PPoincare} with $p = 1$, but are generally not integer-dimensional.

The main result of our paper shows that, in complete doubling metric measure spaces, the weak $1$-Poincar\'e inequality implies the existence of PCs. We achieve this by passing through an intermediate notion, the \emph{generalised pencils of curves}.

\begin{definition}[Generalised pencils of curves]\label{GPC} The space $(X,d,\mu)$ admits \emph{generalised pencils of curves} (GPC in short) if there exists a constant $C_{0} \geq 1$ with the following property. For all distinct $s,t \in X$, there exists a normal $1$-current $T$ on $X$ (in the sense of Ambrosio and Kirchheim) satisfying the following three properties:
\begin{itemize}
\item[(P1)] $\partial T = \delta_{t} - \delta_{s}$,
\item[(P2)] $\spt T \subset B(s,C_{0}d(s,t))$, and
\item[(P3)] $\|T\| = w \, d\mu$ with
\begin{displaymath} |w(y)| \leq C_{0}\left[\frac{1}{\Theta(s,d(s,y))} + \frac{1}{\Theta(t,d(t,y))}\right] \qquad \text{for } \mu\text{-a.e. } y \in X. \end{displaymath}
\end{itemize}
\end{definition}

The main novelty of the present paper is the following result:
\begin{thm}\label{main} Let $(X,d,\mu)$ be complete and doubling. Then $X$ satisfies \eqref{PPoincare} with $p = 1$ if and only if $X$ admits generalised pencils of curves. \end{thm}

It turns out that the existence of GPCs implies the existence of PCs. This is a consequence of a recent decomposition result for normal $1$-currents, due to Paolini and Stepanov \cite{MR2984069}. Combined with Theorem \ref{main}, we obtain the following characterisation.

\begin{thm}\label{mainPCs} Let $(X,d,\mu)$ be complete and doubling. Then $X$ satisfies \eqref{PPoincare} with $p = 1$ if and only if $X$ admits pencils of curves. \end{thm}

\begin{remark} The first version of this paper only contained Theorem \ref{main}, as we were not aware of the decomposition result of Paolini and Stepanov. Shortly afterwards, Theorem \ref{mainPCs} was obtained by Durand-Cartagena, Eriksson-Bique, Korte, and Shanmugalingam \cite[Theorem 3.7]{2018arXiv180903861D}. Their proof uses the modulus of curve families instead of metric currents. So, Theorem \ref{mainPCs} first appeared in \cite{2018arXiv180903861D}. \end{remark}

The structure of the paper is the following. In Section \ref{currentBackground}, we briefly recall the definition of, and some basic concepts related to, the metric currents of Ambrosio and Kirchheim. Then, in Section \ref{twoToOne}, we prove the easy ``if'' implication
\begin{displaymath} \text{GPCs} \quad \Longrightarrow \quad \text{weak } 1\text{-PI} \end{displaymath}
of Theorem \ref{main}, mostly using classical methods in metric analysis. We chose to retain a direct proof of this implication, as it indicates how GPCs can be applied in practice. Another proof would be
\begin{displaymath} \text{GPCs} \quad \Longrightarrow \quad \text{PCs} \quad \Longrightarrow \quad \text{weak } 1\text{-PI}, \end{displaymath}
where the first implication follows from Paolini and Stepanov's work. As explained above, the second implication follows from \cite[Theorem B.15] {MR1414889} in the $Q$-regular case, and in full generality from \cite[Chapter 4]{MR1800917}.

 Section \ref{oneToTwo} is the core of the paper, containing the proof of the ``only if'' implication of Theorem \ref{main}. In short, the idea is to translate the problem of finding currents in $(X,d,\mu)$ to finding ``network flows'' in certain graphs derived from $\delta$-nets in $X$. The existence of such flows is guaranteed by the famous \emph{max flow - min cut} theorem of Ford and Fulkerson \cite{MR2729968}. Then, the main task will be to verify that there are no ``small cuts'' in the graph, and this can be done by using the weak $1$-Poincar\'e inequality. Finally, in Section \ref{s:final}, we explain how to deduce Theorem \ref{mainPCs} from Theorem \ref{main} using the results of Paolini and Stepanov.

\subsection{Basic notation} Open balls in a metric space $(X,d)$ will be denoted by $B(x,r)$, with $x \in X$ and $r > 0$. A \emph{measure} on $(X,d)$ will always refer to a Borel measure $\mu$ with $\mu(B(x,r)) < \infty$ for all balls $B(x,r) \subset X$. The notation $A \lesssim B$ means that there exists a constant $C \geq 1$ such that $A \leq CB$: the constant $C$ will typically depend on the "data" of the ambient space, for example the doubling constant of $\mu$, or the constant in the Poincar\'e inequality \eqref{PPoincare} (whenever \eqref{PPoincare} is assumed to hold). The two-sided inequality $A \lesssim B \lesssim A$ is abbreviated to $A \sim B$.

\section{Acknowledgements}

T.O. is grateful to David Bate and Sean Li for a week of discussions in the summer of 2018, which greatly inspired this project. We are particularly grateful to David Bate for pointing out that the main result might work in all doubling spaces, and not just $Q$-regular ones.

\section{Background on currents}\label{currentBackground}

The main result in the paper mentions currents in metric spaces, so we include here a brief introduction. We claim no originality for anything in this section. We use the definition of metric currents given by Ambrosio and Kirchheim, see Definition 3.1 in \cite{MR1794185}.

Let $X$ be a complete metric space, and let $\textup{Lip}(X)$ and $\textup{Lip}_{b}(X)$ be the families of Lipschitz, and bounded Lipschitz functions on $X$. For $k \geq 1$, Let $\calD^{k}(X) := \Lip_{b}(X) \times [\Lip(X)]^{k}$. We typically denote the $(k + 1)$-tuples in $\calD^{k}(X)$ by $(f,\pi_{1},\ldots,\pi_{k})$. Following Definition 2.2 in \cite{MR1794185}, we consider subadditive, positively $1$-homogenous functionals $T \colon \calD^{k}(X) \to \R$. These are denoted by $MF_{k}(X)$. We say $T \in MF_{k}(X)$ has \emph{finite mass}, if there exists a finite Borel measure $\nu$  on $X$ such that
\begin{equation}\label{mass} |T(f,\pi_{1},\ldots,\pi_{k})| \leq \prod_{j = 1}^{k} \textup{Lip}(\pi_{j}) \int_{X} |f| \, d\nu, \qquad (f,\pi_{1},\ldots,\pi_{k}) \in \calD^{k}(X). \end{equation}
For $k = 0$, the correct interpretation of \eqref{mass} is $|T(f)| \leq \int |f| \, d\nu$. As in Definition 2.6 in \cite{MR1794185}, the minimal measure $\nu$ satisfying \eqref{mass} is denoted by $\|T\|$ (this is well-defined, as discussed below (2.2) in \cite{MR1794185}).

A \emph{$k$-dimensional current}, or just a \emph{$k$-current}, is then a $(k + 1)$-multilinear functional $T \in MF_{k}(X)$ with finite mass, satisfying a few additional requirements which we will not need explicitly, see Definition 3.1 in \cite{MR1794185}. If $T$ is a $k$-current, so that $\|T\|$ is a finite Borel measure, then bounded Lipschitz functions are dense in $L^{1}(X,\|T\|)$, and in particular the space of bounded Borel functions $B(X)$ equipped with the $L^{1}(\|T\|)$-norm. This fact, and \eqref{mass}, together imply that $T$ has a canonical extension to $B(X) \times [\Lip(X)]^{k}$, which we also denote by $T$.

We review a few basic concepts related to currents.

\begin{definition}[Support] The support $\spt(T)$ of a $k$-current $T$ is the usual measure-theoretic support of $\|T\|$, namely
\begin{displaymath} \spt \|T\| = \{x \in X : \|T\|(B(x,r)) > 0 \text{ for all } r > 0\}. \end{displaymath}
\end{definition}


\begin{definition}[Boundary] Let $T \in MF_{k}(X)$, $k \geq 1$. Then $\partial T \in MF_{k - 1}(X)$ is the functional defined by
\begin{displaymath} \partial T(f,\pi_{1},\ldots,\pi_{k - 1}) = T(1,f,\pi_{1},\ldots,\pi_{k}). \end{displaymath}
A $k$-current $T$ is called \emph{normal}, if $\partial T$ is a $(k - 1)$-current, in particular, $\partial T$ has finite mass.
\end{definition}

\begin{definition}[Subcurrent]\label{subcurrent}
A $k$-current $S$ is a subcurrent of a $k$-current $T$, denoted $S\leq T$, if
\begin{displaymath}
\|T-S\|(X) + \|S\|(X) \leq \|T\|(X).
\end{displaymath}
\end{definition}

\begin{definition}[Cycle]
A $k$-current $C$ is a cycle of a $k$-current $T$ if $C\leq T$ and $\partial C =0$.
\end{definition}

\begin{definition}[Acyclic current]\label{d:acyclic} A current $T$ is acyclic if $C=0$ is its only cycle.
\end{definition}

\begin{definition}[Push-forward]\label{pushforward} Let $X,Y$ be complete metric spaces, and let $\varphi \colon X \to Y$ be Lipschitz. For $T \in MF_{k}(X)$, we define the functional $\varphi_{\sharp}T \in MF_{k}(Y)$ by
\begin{displaymath} \varphi_{\sharp}T(g,\pi_{1},\ldots,\pi_{k}) = T(g \circ \varphi,\pi_{1} \circ \varphi,\ldots,\pi_{k} \circ \varphi). \end{displaymath}
\end{definition}


{The following is a special instance of Definition 2.5 in \cite{MR1794185}.
 \begin{definition}[Restriction]\label{d:restr} Let $X$ be a complete metric space, $T \in MF_{1}(X)$, and $g\in \mathrm{Lip}_b(X)$. Then we define an element $ T\lfloor_{ g} \in MF_1(X)$ by setting
 \begin{displaymath}
 T\lfloor_{ g}(f,\pi_1):= T(fg,\pi_1).
 \end{displaymath}
\end{definition}
If $T$ is a $1$-current, then Definition \ref{d:restr} can be extended to $g\in B(X)$ using the canonical extension of $T$ to $B(X) \times \mathrm{Lip}(X)$, see \cite[p.11]{MR1794185}. Moreover, in that case, $T\lfloor_{ g}$ is again a current, see \cite[p.16 and p.19]{MR1794185}. If $E$ is a Borel subset of $X$ and $g=\chi_E$, we write $T\lfloor_{ E}$ for the restriction $T\lfloor_{ g}$.
We have
\begin{equation}\label{eq:restr_meas}
\left|T\lfloor_{E}(f,\pi_1)\right|= \left|T(f\chi_E,\pi_1)\right|\leq \mathrm{Lip}(\pi_1) \int_{E} |f| d\|T\|
\end{equation}
for all $(f,\pi_1)\in \mathcal{D}^1(X)$.
Since $\chi_E$ is merely Borel but not Lipschitz, \eqref{eq:restr_meas} does not follow directly from the definition of $\|T\|$ given by \eqref{mass}, but it can be deduced by the density argument alluded to earlier, see \cite[(2.3)]{MR1794185}. Finally,  \eqref{eq:restr_meas} and the minimality of $\|T\lfloor_{E}\|$ show that
\begin{displaymath}
\mathrm{spt} \left(T\lfloor_{E}\right)\subseteq \mathrm{spt}\left( \|T\|\lfloor_E\right) \subseteq \overline{E}.
\end{displaymath}
}

We record the following lemma, which follows from general measure theory:

\begin{lemma}\label{regularityOfMass} Assume that $T$ is a $k$-current on a $\sigma$-compact metric space $X$. Then, for any Borel set $B \subset X$ and any $\epsilon > 0$, there exists a compact set $K \subset B$ with $\|T\|(B \setminus K) < \epsilon$.
\end{lemma}

\begin{proof} By assumption $\|T\|$ is a finite Borel measure. The claim now follows from \cite[Theorem 1.10]{zbMATH01249699}, and the Note directly below it. \end{proof}

We next describe a simple example, which will be useful later on.
\begin{ex}\label{intervalExample} Given a non-degenerate interval $[a,b] \subset \R$ we may define the $1$-current $\llbracket a,b \rrbracket$ as follows:
\begin{displaymath} \llbracket a,b \rrbracket(f,\pi) = \int_{a}^{b} f(t)\pi'(t) \, dt, \end{displaymath}
where $(f,\pi) \in \Lip_{b}(\R) \times \Lip(\R)$. This is a particular case of Example 3.2 in \cite{MR1794185}, and it is noted there that $\| \llbracket a,b \rrbracket \| = \calH^{1}|_{[a,b]}$. The boundary of $\llbracket a,b \rrbracket$ is the measure (or $0$-current) $\delta_{b} - \delta_{a}$, as shown by the following computation:
\begin{displaymath} \partial \llbracket a,b \rrbracket(f) = \llbracket a,b \rrbracket(1,f) = \int_{a}^{b} f'(t) \, dt = f(b) - f(a) = \int f \, d[\delta_{b} - \delta_{a}]. \end{displaymath}
Next, consider an isometric embedding $\gamma \colon [a,b] \to X$, where $X$ is any complete metric space. Then $\gamma_{\sharp} \llbracket a,b \rrbracket$ defines a current in $X$ given by (spelling out Definition \ref{pushforward})
\begin{displaymath} \gamma_{\sharp}\llbracket a,b \rrbracket(f,\pi) = \llbracket a,b \rrbracket(f \circ \gamma, \pi \circ \gamma) = \int_{a}^{b} f(\gamma(t))(\pi \circ \gamma)'(t) \, dt. \end{displaymath}
It is noted in \cite[(2.6)]{MR1794185}, and in the discussion directly below, that
\begin{equation}\label{form28} \|\gamma_{\sharp}\llbracket a,b \rrbracket \| = \gamma_{\sharp} \llbracket a,b \rrbracket = \gamma_{\sharp}(\calH^{1}\lfloor_{[a,b]}) = \calH^{1}\lfloor_{\gamma([a,b])}. \end{equation}
The last equation follows from the isometry assumption. Finally, because boundary and push-forward commute by \cite[(2.1)]{MR1794185}, we have
\begin{equation}\label{form27} \partial (\gamma_{\sharp} \llbracket a,b \rrbracket) = \gamma_{\sharp}(\partial \llbracket a,b \rrbracket) = \delta_{\gamma(b)} - \delta_{\gamma(a)}. \end{equation}
\end{ex}
We end the section by recalling (a slightly simplified) version of the compactness theorem for normal currents. The original reference, and the full version of the theorem, is \cite[Theorem 5.2]{MR1794185}.
\begin{thm}[Compactness]\label{compactness}
Let $(T_{n})_{n \in \N}$ be a sequence of normal $k$-currents with
\begin{displaymath} \sup_{n} \left( \|T_{n}\|(X) + \|\partial T_{n}\|(X) \right) < \infty, \end{displaymath}
and such that $\spt T_{n} \subset K$ for some fixed compact set $K \subset X$, for all $n \in \N$. Then there exists a subsequence $(T_{n_{m}})_{km\in \N}$, and a normal $k$-current $T$ supported on $K$, such that
\begin{displaymath} \lim_{m \to \infty} T_{n_{m}}(f,\pi_{1},\ldots,\pi_{k}) = T(f,\pi_{1},\ldots,\pi_{k}), \qquad (f,\pi_{1},\ldots,\pi_{k}) \in \Lip_{b}(X) \times [\Lip(X)]^{k}. \end{displaymath}
\end{thm}

\section{Generalised pencils of curves imply the $1$-Poincar\'{e} inequality}\label{twoToOne}

In this section we prove that if $X$ as in Theorem \ref{main} admits generalised pencils of curves, then it supports a weak $1$-Poincar\'{e} inequality.
It is well known, see for instance \cite[Theorem 8.1.7]{MR3363168}, that doubling metric measure spaces which support a Poincar\'{e} inequality can be characterised in terms of the validity of pointwise inequalities between functions and their upper gradients. We will use the existence of GPCs to derive such an inequality between an arbitrary Lipschitz function $u:X \to \mathbb{R}$ and its (upper) \emph{pointwise Lipschitz constant}
\begin{displaymath} x \mapsto \textup{Lip}(u,x) := \limsup_{r \to 0} \sup_{d(y,x) \leq r} \frac{|u(x) - u(y)|}{r}. \end{displaymath}
The desired inequality will be based on the following lemma.

\begin{lemma}\label{lemma1} Let X be a complete $\sigma$-compact metric space and let $u \colon X \to \R$ be a Lipschitz function. Then, for any $1$-current $T$, we have
\begin{displaymath} |\partial T(u)| \leq \int \textup{Lip}(u,x) \, d\|T\|(x). \end{displaymath}
\end{lemma}

\begin{proof} The idea of the proof is the following: By definition of $\partial T$ and since $\|T\|$ is a finite Borel measure on $X$ satisfying \eqref{mass}, we know that
\begin{displaymath}
|\partial T (u)| = |T(1,u)| \leq  \int_X \mathrm{Lip}(u) d\|T\|(x),
\end{displaymath}
where $\mathrm{Lip}(u)$ is the \emph{Lipschitz constant} of $u$, that is, the smallest constant $L\in [0,\infty)$ for which $|u(x)-u(y)|\leq L d(x,y)$ holds for all $x,y\in X$.
The desired inequality in Lemma \ref{lemma1} is similar, but $\mathrm{Lip}(u)$ is replaced by the \textbf{pointwise} Lipschitz constant $\mathrm{Lip}(u,\cdot)$. To achieve this, we will essentially decompose $X$ into pieces where $\mathrm{Lip}(u,\cdot)$ is almost constant.

We now turn to the details. Write
\begin{displaymath} E := \{x \in X : \textup{Lip}(u,x) > 0\} \quad \text{and} \quad Z := \{x \in X : \textup{Lip}(u,x) = 0\}. \end{displaymath}
We perform countable decompositions of the sets $E$ and $Z$. Consider
\begin{displaymath} u_{R}(x) := \sup_{0 < r \leq R} \sup_{d(y,x) \leq r} \frac{|u(x) - u(y)|}{r}, \end{displaymath}
fix $\epsilon > 0$, and define
\begin{displaymath} E_{\delta,j} := \{x \in X : (1 + \epsilon)^{j} < u_{R} \leq (1 + \epsilon)^{j + 1} \text{ for all } R \leq \delta\}, \quad j \in \Z, \: \delta > 0 \end{displaymath}
and
\begin{displaymath} Z_{\delta,j} := \{x \in X : u_{R} \leq 2^{-j} \text{ for all } R \leq \delta\}, \quad j \in \N, \: \delta > 0. \end{displaymath}
Note that $\mathrm{Lip}(u,x)= \lim_{R\to 0} u_R(x)$ and that for every $j \in \Z$ fixed, the sequences $(E_{1/i,j})_{i\in \mathbb{N}}$ and
$(Z_{1/i,j})_{i\in \mathbb{N}}$ are nested.
Then, we can write
\begin{equation}\label{form4} X  = E \cup Z \subset \bigcup_{i \in \N} \bigcup_{j \in \Z} E_{1/i,j} \cup \bigcap_{j \in \N} \bigcup_{i \in \N} Z_{1/i,j}. \end{equation}
Fix $x \in X$, $i \in \N$. Notice that the restriction of $u$ to the set
\begin{displaymath} B(x,1/(2i)) \cap E_{1/i,j} =: E^{x}_{1/(2i),j} \end{displaymath}
is $(1 + \epsilon)^{j + 1}$-Lipschitz since, if $y,z \in E^{x}_{1/(2i),j}$, then
\begin{equation*} y,z \in E_{1/i,j} \quad \text{and} \quad d(y,z) \leq \frac{1}{i}, \end{equation*}
which implies that
\begin{equation}\label{form3} \frac{|u(y) - u(z)|}{d(y,z)} \leq u_{1/i}(y) \leq (1 + \epsilon)^{j + 1}. \end{equation}
Moreover,
\begin{equation}\label{form2} \textup{Lip}(u,y) \geq (1 + \epsilon)^{j}, \qquad y \in E_{1/(2i),j}^{x}. \end{equation}

A similar argument applies if $x \in X$, and $y,z \in B(x,1/(2i)) \cap Z_{1/i,j} =: Z_{1/(2i),j}^{x}$. Then the conclusion is that the restriction of $u$ to the set $Z_{1/(2i),j}^{x}$ is $2^{-j}$-Lipschitz, but one does not care about (and cannot have) the lower bound \eqref{form2}.

Since $(X,d)$ is $\sigma$-compact, it is separable and hence we can pick a countable dense subset $\{x_n\}_{n\in\mathbb{N}}\subseteq X$. Then for arbitrarily small $\delta \in [0,1/2)$, the sets
\begin{displaymath}
\left\{E_{1/(2i),j}^{x_n}:\; i\in \mathbb{N},j\in\mathbb{Z},n\in\mathbb{N}\right\} \cup \left\{Z_{1/(2i),j}^{x_n}:\;i\in\mathbb{N}, j\in \mathbb{N}\text{ with }j\geq - \log_2 \delta,n\in\mathbb{N}\right\}
\end{displaymath}
constitute a countable cover of $X$ by Borel sets. Using this cover, we can easily construct a countable \textbf{disjoint} cover of $X$ by Borel sets $\{E_i\}_{i\in \mathbb{N}}$ and $\{Z_i\}_{i\in\mathbb{N}}$ such that the function $u$ can be decomposed as
\begin{equation}\label{form5} u = u\chi_{E} + u\chi_{Z} = \sum_{i \in \N} u\chi_{E_{i}} + \sum_{i \in \N} u\chi_{Z_{i}} \end{equation}
where
\begin{equation}\label{eq:e_i} u|_{E_{i}} \text{ is } L_{i}\text{-Lipschitz} \quad \text{and} \quad \textup{Lip}(u,x) \geq (1 - \epsilon)L_{i} \text{ for } x \in E_{i} \end{equation}
for some finite constants $L_{i} > 0$, and
\begin{displaymath} u|_{Z_{i}} \text{ is } \delta\text{-Lipschitz}, \end{displaymath}
where $\delta > 0$ can be taken arbitrarily small.
Here we have used that the properties
\eqref{form3} and \eqref{form2} (and their counterparts for $Z_{1/(2i),j}$) are preserved under taking subsets.

 Moreover, Lemma \ref{regularityOfMass} allows us to remove for every $i\in \mathbb{N}$ a Borel set $N_i$ from $E_i$ (or similarly $Z_{i}$) such that
 \begin{displaymath}
 E_i \setminus N_i\text{ is compact  }\quad\text{and}\quad \|T\|(N_i) < \frac{\epsilon}{2^{i+1}},
 \end{displaymath}
 Thus, we may assume that the sets $\{E_i\}_{i\in\mathbb{N}}\cup \{Z_i\}_{i\in\mathbb{N}}$ are compact, if we replace \eqref{form5} by a decomposition
\begin{displaymath} u = \sum_{i \in \N} u\chi_{E_{i}} + \sum_{i \in \N} u\chi_{Z_{i}} + u\chi_{N}, \end{displaymath}
where $N \subset X$ is a Borel set with $\|T\|(N) < \epsilon$.

Next, we use the McShane extension theorem to find Lipschitz functions $u^{E}_{i},u^{Z}_{i} \colon X \to \R$ such that
\begin{displaymath} u^{E}_{i}|_{E_{i}} = u|_{E_{i}} \quad \text{and} \quad u^{Z}_{i}|_{Z_{i}} = u|_{Z_{i}} \end{displaymath}
and
\begin{displaymath} u_{i}^{E} \text{ is } L_{i}\text{-Lipschitz} \quad \text{and} \quad u_{i}^{Z} \text{ is } \delta\text{-Lipschitz}. \end{displaymath}
Then
\begin{displaymath} u = \sum_{i \in \N} u_{i}^{E}\chi_{E_{i}} + \sum_{i \in \N} u_{i}^{Z}\chi_{Z_{i}} + u\chi_{N}, \end{displaymath}
and we can write
\begin{equation}\label{eq:decomp} |\partial T(u)| = |T(1,u)| \leq \sum_{i \in \N} |T(\chi_{E_{i}},u)| + \sum_{i \in \N} |T(\chi_{Z_{i}},u)| + |T(\chi_{N},u)| \end{equation}
Since $\textup{Lip}(u) < \infty$, we have $|T(\chi_{N},u)| \leq \textup{Lip}(u)\|T\|(N) < \mathrm{Lip}(u) \epsilon$.
We now estimate the terms involving $\chi_{E_i}$.
By definition of the restriction operation, see Definition \ref{d:restr} and the comment below it, we can write
\begin{equation}\label{eq:1_sum}
\left|T(\chi_{E_{i}},u)\right| =  \left|T\lfloor_{E_{i}}(1,u)\right|.
\end{equation}
Recall that $u|_{E_i}= u_i^E$. This is useful information since, according to \cite[(3.6)]{MR1794185}, the values of a $1$-current agree on $(f,\pi_1)$ and $(f',\pi_1')$ whenever
$f=f'$ and $\pi_1= \pi_1'$ on the support of $T$. {Using that $\mathrm{spt}\left( T\lfloor_{E_i}\right)\subseteq E_i$}, we apply this fact to the current $T\lfloor_{E_{i}}$ and the pairs $(f,\pi_1)=(1,u)$ and $(f',\pi_1')=(1,u_i^E)$, $i\in \mathbb{N}$. This shows that
\begin{equation}\label{eq:2_sum}
\left|T\lfloor_{E_{i}}(1,u)\right| =  \left|T\lfloor_{E_{i}}(1,u_i^E)\right|
\end{equation}
Finally,  by \eqref{eq:restr_meas} and the property \eqref{eq:e_i} of $u_i^E$, it holds for every $i\in \mathbb{N}$ that
\begin{equation}\label{eq:3_sum}
\left|T\lfloor_{E_{i}}(1,u_i^E)\right| \leq \int_{E_i} L_i d \|T \| \leq \frac{1}{1-\epsilon} \int_{E_i} \mathrm{Lip}(u,x) d\|T\|(x).
\end{equation}
 Combining \eqref{eq:1_sum}, \eqref{eq:2_sum}, and \eqref{eq:3_sum}, and using the pairwise disjointedness of the sets $E_i$, $i\in \mathbb{N}$, we conclude that
\begin{displaymath}
\sum_{i \in \mathbb{N}} \left|T(\chi_{E_{i}},u)\right| \leq \frac{1}{1-\epsilon} \int_{X} \mathrm{Lip}(u,x) d\|T\|(x)
\end{displaymath}
%
Similar considerations give
\begin{displaymath} \sum_{i \in \N} |T(\chi_{Z_{i}},u)| \leq \delta\|T\|(Z). \end{displaymath}

Letting $\epsilon \to 0$ and $\delta \to 0$ in \eqref{eq:decomp} completes the proof of  Lemma \ref{lemma1}.  \end{proof}

We next apply Lemma \ref{lemma1} to deduce the validity of a weak $1$-Poincar\'{e} inequality from the existence of GPCs.

\begin{proof}[Proof of ``if'' implication in Theorem \ref{main}] By a result of Keith \cite{MR2013501}, see also Theorem 8.4.2 in \cite{MR3363168}, it suffices to verify the Poincar\'e inequality for \emph{a priori} Lipschitz continuous functions $u$ and for the pointwise Lipschitz constant $\rho = \mathrm{Lip}(u,\cdot)$ instead of arbitrary upper gradients. So, let $u \colon X \to \R$ with $\textup{Lip}(u) < \infty$.

Recalling that
\begin{displaymath}
\Theta(x,r)= \frac{\mu(B(x,r))}{r},\quad x\in X,\;r>0,
\end{displaymath}
we will first check that
\begin{equation}\label{form1} |u(t) - u(s)| \lesssim \int_{B(s,C_{0}d(s,t))} \frac{\textup{Lip}(u,y)}{\Theta(s,d(s,y))} + \frac{\textup{Lip}(u,y)}{\Theta(t,d(t,y))} \, d\mu(y)\end{equation}
for distinct points $s,t \in X$. Start by fixing such points $s,t$, let $T$ be a GPC joining $s$ to $t$, and recall that $\spt T \subset B(s,C_{0}d(s,t))$. Then,
\begin{align*} |u(t) - u(s)| = |\partial T(u)|&  \leq \int \textup{Lip}(u,y) \, d\|T\|(y)\\
& \leq C_0 \int_{B(s,C_{0}d(s,t))} \frac{\textup{Lip}(u,y)}{\Theta(s,d(s,y))} + \frac{\textup{Lip}(u,y)}{\Theta(t,d(t,y))} \, d\mu(y), \end{align*}
using Lemma \ref{lemma1} and the property (P3) of the GPC $T$. (Since a complete doubling metric space is proper, see \cite[Lemma 4.1.14]{MR3363168}, and a proper metric space is $\sigma$-compact, the space $(X,d)$ satisfies the assumptions of Lemma  \ref{lemma1}). This proves \eqref{form1}, which is (almost) a well-known sufficient condition for the weak $1$-Poincar\'{e} inequality. The only technicality here is that we only know \eqref{form1} for the particular upper gradient $\Lip(u,\cdot)$. To complete the proof, we will now briefly argue that that this suffices to imply the weak $1$-Poincar\'e inequality in full generality.

Indeed, \cite[Theorem 9.5]{MR1800917} lists several  conditions that imply  weak Poincar\'e inequalities in doubling spaces (see also the references in \cite{MR1800917}). Our estimate \eqref{form1} shows that condition (2) in \cite[Theorem 9.5]{MR1800917} holds for $p=1$, $\mu$,  $u$ Lipschitz, the particular upper gradient $\rho = \mathrm{Lip}(u,\cdot)$, $C_2=C_3 = C_0$. It then follows from the proof in \cite{MR1800917} that also condition (3) in the theorem holds for the same pair $(u,\rho)$ (by this, we mean that to obtain condition (3) for $u$ and $\rho$, one only needs to have condition (2) for $u$ and $\rho$, and no other upper gradients). We rephrase condition (3) in a slightly peculiar manner for future application: there exists a constant $C\geq 1$ (again depending on $C_{0}$) such that if
$2B:= B(z,2r)$ is any ball in $X$, and $x,y\in 2B$, then
then
\begin{equation}\label{eq:(3)} |u(x) - u(y)| \lesssim d(x,y) \left( M_{Cd(x,y)}\rho(x) + M_{Cd(x,y)}\rho(y) \right). \end{equation}
Here $M_{R}$ is the restricted maximal function
\begin{displaymath} M_{R}\rho(x) = \sup_{r < R} \frac{1}{\mu(B(x,r))} \int_{B(x,r)} \rho(y) \, d\mu (y), \qquad R > 0. \end{displaymath}
Next we need to verify that inequality \eqref{eq:(3)} implies that the very same pair $(u,\rho)$ satisfies the weak $1$-Poincar\'{e} inequality.
To this end, we apply Theorem 8.1.18 in \cite{MR3363168} (originally due to Haj{\l}asz \cite{MR1401074}) with $h = M_{4Cr}\rho$ and $Q= \log_2 C_{\mu}$ for the doubling constant $C_{\mu}$ of $\mu$, to deduce that
\begin{displaymath} \frac{1}{\mu(B)} \int_{B} |u - u_{B}| \, d\mu^Q \lesssim r \left( \frac{1}{\mu(2B)} \int_{2B} [M_{4Cr}\rho]^{\tfrac{Q}{Q + 1}} \, d\mu \right)^{\tfrac{Q + 1}{Q}}. \end{displaymath}
Finally, following verbatim the argument on p. 224 of \cite{MR3363168},
we conclude that
\begin{align*} \frac{1}{\mu(B)} \int_{B} |u - u_{B}| \, d\mu &\lesssim  r \left( \frac{1}{\mu(2B)} \int_{2B} [M_{4Cr}\rho]^{\tfrac{Q}{Q + 1}} \, d\mu \right)^{\tfrac{Q + 1}{Q}}\\
&\lesssim  r \left(\fint_{2CB} [M_{4Cr}\rho]^{\tfrac{Q}{Q + 1}} \, d\mu \right)^{\tfrac{Q + 1}{Q}}\\
&\lesssim \fint_{6CB} \rho \;\mathrm{d}\mu. \end{align*}
 Hence the inequality \eqref{PPoincare} holds with  $p=1$ and $\lambda = 6C$  (depending on $C_0$) for all open balls $B\subset X$ and all pairs $(u,\rho)$, where $u:X \to \mathbb{R}$ is Lipschitz and $\rho=\mathrm{Lip}(u,\cdot)$. According to \cite[Theorem 8.4.2]{MR3363168}, this shows that $(X,d,\mu)$ supports the weak $1$-Poincar\'e inequality.
\end{proof}

\section{From $1$-Poincar\'{e} to generalised pencils of curves}\label{oneToTwo}

\subsection{An initial reduction} The main effort in the rest of the paper consists of proving the following statement:

\begin{thm}\label{main2} Let $(X,d,\mu)$ be a complete doubling \textbf{and geodesic} metric measure space. If $X$ supports a $1$-Poincar\'{e} inequality, then it supports generalised pencils of curves.
 \end{thm}

The assumptions in Theorem \ref{main2} are superficially stronger than in the remaining implication of Theorem \ref{main}, so we start by briefly discussing how Theorem \ref{main} reduces to the special case in Theorem \ref{main2}.

\begin{proof}[Proof of the ``only if'' part of Theorem \ref{main}, assuming Theorem \ref{main2}] By \cite[Corollary 8.3.16]{MR3363168}, if the complete doubling space $(X,d,\mu)$ supports a weak $1$-Poincar\'e inequality, then $d$ is biLipschitz equivalent to a geodesic metric $g$. Further, by \cite[Lemma 8.3.18]{MR3363168}, the space $(X,g,\mu)$ still supports a weak $1$-Poincar\'e inequality.
In fact, since $(X,g)$ is geodesic,
it  follows from \cite[Remark 9.1.19]{MR3363168} that $(X,g,\mu)$ even satisfies the weak $1$-Poincar\'{e} inequality \eqref{PPoincare} with constant $\lambda =1$; this is often called the $1$-Poincar\'e inequality (without the attribute "weak").

We can thus apply Theorem \ref{main2}  to $(X,g,\mu)$ in order to find a GPC between any pair of distinct points $s,t \in X$. Then $T$ is also a GPC joining $s$ to $t$ in $(X,d,\mu)$, since the conditions (P1)-(P3) in Definition \ref{GPC} are obviously invariant under bi-Lipschitz changes of metric. The least obvious is (P3), where one needs to recall that $\mu$ is a doubling measure, whence $\Theta_{d}(s,d(s,y)) \sim \Theta_{g}(s,g(s,y))$ and $\Theta_{d}(t,d(t,y)) \sim \Theta_{g}(t,g(t,y))$ for all $y \in X$. \end{proof}

As in the assumptions of Theorem \ref{main2}, we now suppose that $(X,d,\mu)$ is a complete geodesic doubling metric measure space supporting the Poincar\'e  inequality  \eqref{PPoincare} with $p=1$ and  $\lambda =1$.

\subsection{Proof of Theorem \ref{main2}}

Fix two points $s,t \in X$, and write $B_{0} := B(s,C_{0}d(s,t))$ for the ball inside (the closure of) which we should find the current $T$ as in Definition \ref{GPC}; the constant $C_{0} \geq 1$ will be specified later, and its size only depends on the data of $(X,d,\mu)$, such as the doubling constant of $\mu$, and the constant in the Poincar\'e inequality. We find the current $T$ by initially constructing a sequence of approximating currents, each of them a sum of finitely many currents of the form discussed in Example \ref{intervalExample}. We start by defining a sequence of covers of $X$ by balls. For $n \in N$, write $r_{n} := 2^{-n}$, and let $X_{n} \subset B_{0}$ be an $r_{n}$-net, that is, some maximal family of points $X_{n} \subset B_{0}$ satisfying $d(x,x') \geq r_{n}$ for all distinct $x,x' \in X_{n}$. We assume that $r_{n}$ is far smaller than $d(s,t)$, and we require that
\begin{equation}\label{form7} \{s,t\} \subset X_{n}. \end{equation}
We note that the collection of {open} balls $\calB_{n} := \{B(x,2r_{n}) : x \in X_{n}\}$ is now a cover of $B_{0}$. In fact, already the balls $B(x,r_{n})$ would be a cover of $B_{0}$: by the maximality of $X_{n}$, for every $y \in B_{0}$ there exists $x \in X_{n}$ such that \begin{equation}\label{coverProperty} y \in B(x,r_{n}). \end{equation}
Moreover, every ball $B \in \calB_{n}$ only has boundedly many ``neighbours'':
\begin{equation}\label{neighbours} \card \{B' \in \calB_{n} : B \cap B' \neq \emptyset\} \lesssim 1. \end{equation}
This follows by using the doubling property  of $\mu$ and the consequential relative lower volume decay (see \cite[Lemma 8.1.13]{MR3363168}), and noting that the balls $B(x,r_{n}/2)$, $x \in X_{n}$, are disjoint.

We will now construct a current $T_{n}$ supported in $\Omega_n = B(s, C_{0}d(s,t) + C r_n)$ for suitable constants $C_{0},C \geq 1$ (depending on the constants in the $1$-Poincar\'e inequality). The current $T_{n}$ will be constructed using the \emph{max flow min cut} theorem from graph theory, and a subsequence of the currents $T_{n}$ will eventually be shown, using the compactness theorem for normal currents, to converge to the desired current $T$ supported on $\bar{B}_{0}$.

\subsection{Graphs and flows}

To apply the max flow min cut theorem, we need to define a graph $G_{n} = (\mathcal{V}_{n},E_{n})$ associated to our problem. We set $\mathcal{V}_{n} := X_{n}$, and
\begin{displaymath} E_{n} := \{(x,x') \in \mathcal{V}_{n} \times \mathcal{V}_{n} : x \neq x' \text{ and } B(x,2r_{n}) \cap B(x',2r_{n}) \neq \emptyset\}. \end{displaymath}
Note that $(x,x') \in E_{n}$ if and only if $(x',x) \in E_{n}$, and that the maximum degree of any vertex is uniformly bounded by \eqref{neighbours}. We also define a \emph{capacity function} $c_{n} \colon E_{n} \to \Q_{+}$ satisfying
\begin{equation}\label{form9} c_{n}(x,x') \sim \frac{\Theta(x,r_{n})}{\Theta(s,d(s,x))} + \frac{ \Theta(x',r_{n})}{\Theta(t,d(t,x'))}, \quad (x,x') \in E_{n}, \end{equation}
if $\{x,x'\} \cap \{s,t\} = \emptyset$, and $c_{n}(x,x') \sim 1$ otherwise. Here $\Theta(x,r)$ is the ($1$-dimensional) density
\begin{displaymath} \Theta(x,r) = \frac{\mu(B(x,r))}{r}, \qquad x \in X, \: r > 0. \end{displaymath}
We do not specify the values of $c_n(x,x')$ more precisely: we will only use that $c_{n}(x,y)$ is a rational number within a constant multiple of the right hand side of \eqref{form9}. Note that $c_n(x,x') \sim c_n(x',x)$ for all $(x,x') \in E_{n}$; in fact, we may as well define $c_{n}(x,x') = c_{n}(x',x)$.

A \emph{flow in $G_{n}$} is a function $f \colon E_{n} \to \R$ satisfying the following three conditions:
\begin{itemize}
\item[(F1)] $f(x,x') = -f(x',x)$ for all $(x,x') \in E_{n}$,
\item[(F2)] $f(\{x\},\mathcal{V}_{n}) = 0$ for all $x \in \mathcal{V}_{n} \setminus \{s,t\}$,
\item[(F3)] $f(x,x') \leq c_{n}(x,x')$ for all $(x,x') \in E_{n}$.
\end{itemize}
Here, and in the sequel, we write
\begin{displaymath} f(\mathcal{U},\mathcal{W}) = \sum_{e \in E_{n}(\mathcal{U},\mathcal{W})} f(e), \end{displaymath}
where $E_{n}(\mathcal{U},\mathcal{W}) = \{(x,x') \in E_{n} : x \in \mathcal{U} \text{ and } x' \in \mathcal{W}\}$, and analogously
\begin{displaymath} c_n(\mathcal{U},\mathcal{W}) = \sum_{e \in E_{n}(\mathcal{U},\mathcal{W})} c_n(e).\end{displaymath}
 The \emph{norm} of a flow $f \colon E_{n} \to \R$ is defined to be the quantity
\begin{displaymath} \|f\| := f(\{s\},\calV_{n}). \end{displaymath}
A \emph{cut} is any pair $(\calS,\calS^{c})$, where $\calS \subset \mathcal{V}_{n}$ is a set with $s \in \mathcal{S}$ and $t \in \mathcal{S}^{c}$. As usual, $\calS^{c}$ denotes the complement of $\calS$ (in $\calV_{n}$). The ``total flow'' of $f$ over any cut $(\calS,\calS^{c})$ equals $\|f\|$:
\begin{equation}\label{normDef} \|f\| = f(\calS,\calS^{c}). \end{equation}
In particular, $\|f\| = f(\calV_{n}\setminus \{t\},\{t\})$. For a proof of \eqref{normDef}, see \cite[Proposition 6.2.1]{MR2729968}.  Consequently, by (F3),
\begin{equation}\label{form8} \|f\| \leq \min_{(\calS,\calS^{c})} c_{n}(\calS,\calS^{c}), \end{equation}
where the $\min$ runs over all cuts $(\calS,\calS^{c})$. In other words, the norm of any flow is bounded from above by the capacity of any cut in the graph.

A well-known theorem in graph theory due to Ford and Fulkerson \cite{MR2729968} states that if $c_{n}$ is integer-valued, then \eqref{form8} is sharp: there exists a flow $f$ with $\|f\| = \min c_{n}(\calS,\calS^{c})$. We learned the theorem from Diestel's graph theory book, see \cite[Theorem 6.2.2]{MR3644391}. Our capacity $c_{n}$ is not integer valued, but since $c_{n}(x,x') \in \Q_{+}$, and the cardinality of $E_{n}$ is finite, we may assume that $c_{n}(x,x') \in \N$ by initially multiplying all quantities by a suitable integer.

The reader should view flows in $G_{n}$ as discrete models for the current $T_{n}$: we will make the connection rigorous in Section \ref{currentSection}. For now, we wish to find a uniform lower bound for the numbers $c_{n}(\calS,\calS^{c})$, where $(\calS,\calS^{c})$ is an arbitrary cut. We claim that
\begin{equation}\label{form10} c_{n}(\calS,\calS^{c}) \gtrsim 1. \end{equation}
To this end, fix a cut $(\calS,\calS^{c})$, and recall that $s \in \calS$ and $t \in \calS^{c}$ by definition. Also by definition,
\begin{align} c_{n}(\calS,\calS^{c}) & = \sum_{(x,x') \in E_{n}(\calS,\calS^{c})} c_{n}(x,x') \notag\\
&\label{form21} \sim \sum_{(x,x') \in E_{n}(\calS,\calS^{c})} \frac{\Theta(x,r_{n})}{\Theta(s,d(s,x))} + \frac{ \Theta(x',r_{n})}{\Theta(t,d(t,x'))}. \end{align}
To be precise, \eqref{form21} only holds if none of the edges in $E_{n}(\calS,\calS^{c})$ start or end in $\{s,t\}$. We may assume this, since, for example, if $(s,x') \in E_{n}(\calS,\calS^{c})$, then $c_{n}(\calS,\calS^{c}) \geq c_{n}(s,x') \gtrsim 1$, and \eqref{form10} follows. In fact, the same argument holds a little more generally: if $(x,x') \in E_{n}(\calS,\calS^{c})$ satisfies $\dist(x,\{s,t\}) \lesssim r_{n}$ or $\dist(x',\{s,t\}) \lesssim r_{n}$, then again $c_{n}(\calS,\calS^{c}) \gtrsim 1$, using the assumption that $\mu$ is doubling. So, without loss of generality, we assume that
\begin{equation}\label{form18} \min\{\dist(x,\{s,t\}),\dist(x',\{s,t\})\} \geq Cr_{n}, \qquad (x,x') \in E_{n}(\mathcal{S},\mathcal{S}^{c}). \end{equation}
where $C \geq 1$ is a suitable large constant to be specified later. If $C \geq 20$, say, then \eqref{form18} has the following consequence:
\begin{equation}\label{form15} d(s,y) \sim d(s,x) \sim d(s,x') \quad \text{and} \quad d(t,y) \sim d(t,x') \sim d(t,x) \end{equation}
for all $y \in \overline{B(x,5r_{n})} \cup \overline{B(x',5r_{n})}$, and all pairs $(x,x') \in E_{n}(\calS,\calS^{c})$, since $d(x,x') \leq 4r_{n}$ for such pairs. Note that \eqref{form15}, combined with the doubling of $\mu$, also allows us to replace the denominators in \eqref{form18} by some comparable quantities, as indicated by \eqref{form15}, for example
\begin{equation}\label{form32} \Theta(t,d(t,x')) \sim \Theta(t,d(t,x)), \qquad (x,x') \in E_{n}(\calS,\calS^{c}). \end{equation}

Evidently, the proof of \eqref{form10} should somehow use our only assumption: the Poincar\'e  inequality \eqref{PPoincare} with $p=1$. To this end, we define a Lipschitz function $u = u_{n} \colon B_{0} \to \R$ associated to the cut $(\calS,\calS^{c})$, using a Lipschitz partition of unity on $B_{0}$, subordinate to the cover $\calB_{n}$. For $x \in \calV_{n}$, let
\begin{displaymath} \phi_{x}(y) = \chi_{B(x,2r_{n})}(y)\cdot \frac{2r_{n} - d(y,x)}{2r_{n}} \quad \text{and} \quad \psi_{x} := \frac{\phi_{x}}{\sum_{x' \in \calV_{n}} \phi_{x'}}. \end{displaymath}
Then
\begin{equation}\label{form22} \{y \in B_{0} : \psi_{x}(y) > 0\} = B_{0} \cap B(x,2r_{n}), \end{equation}
and $\psi_{x}$ and is $(C/r_{n})$-Lipschitz on $B_{0}$: this is easy to check, noting that
\begin{displaymath} \sum_{x' \in \calV_{n}} \phi_{x'}(y) \sim 1, \qquad y \in B_{0}, \end{displaymath}
by \eqref{coverProperty} and the bounded overlap of the balls in $\calB_{n}$. Evidently,
\begin{displaymath} \sum_{x \in \calV_{n}} \psi_{x}(y) = 1, \qquad y \in B_{0}, \end{displaymath}
so the family $\{\psi_{x}\}_{x \in \calV_{n}}$ is the partition of unity on $B_{0}$ we were after.
We set
\begin{displaymath} u := \sum_{x \in \calS} \psi_{x}. \end{displaymath}
Evidently, $u$ takes values in $[0,1]$, and is $L_{n}$-Lipschitz on $B_{0}$ for some $L_{n} \sim 1/r_{n}$. For $y \in X$ and any subset $\calU \subset \calV_{n}$, write
\begin{displaymath} \calU(y) := \{x \in \calU : y \in B(x,2r_{n})\} = \{x \in \calU : \phi_{x}(y) > 0\} \end{displaymath}
Clearly $\calS(y),\calS^{c}(y) \subset \calV_{n}(y)$ and $\calV_{n}(y) = \calS(y) \cup \calS^{c}(y)$ for all $y \in X$.
\begin{lemma}\label{lemma2} For $y \in B_{0}$, we have
\begin{displaymath} u(y) = 1 \quad \Longleftrightarrow \quad \calV_{n}(y) = \calS(y), \end{displaymath}
and
\begin{displaymath} u(y) = 0 \quad \Longleftrightarrow \quad \calV_{n}(y) = \calS^{c}(y). \end{displaymath}
\end{lemma}

\begin{proof} If $u(y) = 1$, then
\begin{displaymath} 1 = \sum_{x \in \calS} \psi_{x}(y) = \sum_{x \in \calS} \frac{\phi_{x}(y)}{\sum_{x' \in \calV_{n}} \phi_{x'}(y)} =  \frac{\sum_{x \in \calS(y)} \phi_{x}(y)}{\sum_{x' \in \calV_{n}(y)} \phi_{x'}(y)}.  \end{displaymath}
Hence
\begin{displaymath} \sum_{x \in \calS(y)} \phi_{x}(y) = \sum_{x' \in \calV_{n}(y)} \phi_{x'}(y), \end{displaymath}
which forces $\calS(y) = \calV_{n}(y)$. The converse implication is clear.

If $u(y) = 0$, then $\psi_{x}(y) = 0$ for all $x \in \calS$, so $\calS(y) = \emptyset$. Consequently, $\calV_{n}(y) \subset \calS^{c}$, as claimed. The converse implication is again clear.
\end{proof}

\begin{cor}\label{cor1} We have $u(s) = 1$ and $u(t) = 0$.
\end{cor}

\begin{proof} This is, in fact, a corollary of Lemma \ref{lemma2} and \eqref{form18}. Start with $s$: if $u(s) < 1$, then $\phi_{x}(s) > 0$ for some $x \in \calS^{c}$, hence $s \in B(x,2r_{n})$ by \eqref{form22}. On the other hand, $s \in \calS$ (by the very definition of a cut), so the fact that $s \in \spt \psi_x\subseteq  B(x,2r_{n})$ implies $(s,x) \in E_{n}(\calS,\calS^{c})$. This contradicts \eqref{form18} as soon as $C\geq 2$, and hence we deduce that $u(s) = 1$.

The treatment of $t$ is essentially symmetric: if $u(t) > 0$, then $\phi_{x}(t) > 0$ for some $x \in \calS$. But since $t \in \calS^{c}$, this implies that $(x,t) \in E_{n}(\calS,\calS^{c})$, again violating \eqref{form18}. \end{proof}

Next, still using Lemma \ref{lemma2}, we investigate where $\textup{Lip}(u,y) = 0$.
\begin{lemma} We have
\begin{equation}\label{lemma3} \{y \in B_{0} : \textup{Lip}(u,y) \neq 0\} \subset \bigcup_{x \in \emph{\textbf{Bd}}(\calS)} B(x,5r_{n}), \end{equation}
where
\begin{displaymath} \emph{\textbf{Bd}}(\calS) := \{x \in \calS : \exists \: x' \in \calS^{c} \text{ such that } (x,x') \in E_{n}\}, \end{displaymath} \end{lemma}

\begin{proof} Pick $y \in B_{0}$ with $\textup{Lip}(u,y) \neq 0$. We claim that this has the following consequence: \emph{for all $x \in \calV_{n}(y)$, the set $\calN(x) := \{x\} \cup \{x' \in \calV_{n} : (x,x') \in E_{n}\}$ intersects both $\calS$ and $\calS^{c}$.}

Assume to the contrary that there is some $x \in \calV_{n}(y)$ with $\calN(x) \subset \calS$ or $\calN(x) \subset \calS^{c}$: we start with the case $\calN(x) \subset \calS$. Pick  $z \in B_{0} \cap B(x,2r_{n})$ arbitrarily, and consider any $x' \in \calV_{n}(z)$. Then $z \in B(x',2r_{n})$ by definition, so
\begin{displaymath} z \in B(x,2r_{n}) \cap B(x',2r_{n}) \quad \Longrightarrow \quad (x,x') \in E_{n} \quad \Longrightarrow \quad x' \in \calS. \end{displaymath}
This shows that $\calV_{n}(z) \subset \calS$, hence $u(z) = 1$ by Lemma \ref{lemma2}. But $z \in B_{0} \cap B(x,2r_{n})$ was arbitrary, so we have inferred that $u \equiv 1$ on the neighbourhood $B_{0} \cap B(x,2r_{n})$ of $y$. In particular $\Lip(u,y) = 0$, a contradiction.

Next, consider the case $\calN(x) \subset \calS^{c}$. As before, pick $z \in B_{0} \cap B(x,2r_{n})$ arbitrarily, and deduce as above that $x' \in \calS^{c}$ for all $x' \in \calV_{n}(z)$. This implies by Lemma \ref{lemma2} that $u(z) = 0$, and hence $u \equiv 0$ on $B_{0} \cap B(x,2r_{n})$. This contradicts $\Lip(u,y) \neq 0$.

Now that we have proven the claim in italics, we finish the proof of the lemma. Fix $y \in B_{0}$ with $\Lip(u,y) \neq 0$, pick any $x \in \calV_{n}(y)$, and assume first that $x \in \calS$. Then there exists $x' \in \calS^{c}$ with $(x,x') \in E_{n}$. Hence $x \in \textbf{Bd}(\calS)$ by definition, and $y \in B(x,2r_{n}) \subset B(x,5r_{n})$, as claimed. Next, if $x \in \calS^{c}$, then we have shown that there exists $x' \in \calS$ with $(x',x) \in E_{n}$. This means that $x' \in \textbf{Bd}(\calS)$. Since $B(x,2r_{n}) \cap B(x',2r_{n}) \neq \emptyset$, we infer that $y \in B(x',5r_{n})$, and the proof is complete. \end{proof}

We extend $u$ to an $L_{n}$-Lipschitz map $X \to \R$ without change in the notation. Then, we note that ($u$, $\textup{Lip}(u,\cdot)$) is a function - upper gradient pair on $X$, and we apply Theorem 9.5 in \cite{MR1800917}, more precisely the implication "$(4) \Longrightarrow (2)$", which requires the space $(X,d)$ to be geodesic. This implication gives the following estimate:
\begin{displaymath} 1 = |u(s) - u(t)| \lesssim \int_{B_{0}} \frac{\Lip(u,y)}{\Theta(s,d(s,y))} + \frac{\Lip(u,y)}{\Theta(t,d(t,y))} \, d\mu(y), \end{displaymath}
assuming that $s,t$ are "deep enough inside" the ball $B_{0}$. This can be arranged by choosing $C_{0} \geq 1$ in the definition of $B_{0}$ large enough. Then, if (by slight abuse of notation) we denote by \textbf{Bd}$(\calS)$ the set on the right hand side of \eqref{lemma3}, we obtain further
\begin{equation}\label{form20} 1 \lesssim \frac{1}{r_{n}} \int_{B_{0} \cap \textbf{Bd}(\calS)} \frac{1}{\Theta(s,d(s,y))} + \frac{1}{\Theta(t,d(t,y))} \, d\mu(y). \end{equation}
Using the definition of \textbf{Bd}$(\calS)$, and recalling \eqref{form15}-\eqref{form32}, we may continue the estimate \eqref{form20} as follows:
 \begin{align*} 1 & \lesssim \frac{1}{r_{n}} \sum_{x \in \textbf{Bd}(\calS)} \frac{\mu(B(x,r_{n}))}{\Theta(s,d(s,x))} + \frac{\mu(B(x,r_{n}))}{\Theta(t,d(t,x))}\\
 & \leq \sum_{(x,x') \in E_{n}(\calS,\calS^{c})} \frac{\Theta(x,r_{n})}{\Theta(s,d(s,x))} + \frac{\Theta(x',r_{n})}{\Theta(t,d(t,x'))} \lesssim c_{n}(\calS,\calS^{c}), \end{align*}
 recalling \eqref{form21} in the last inequality. This proves \eqref{form10}.

 By multiplying the $c_{n}$ by a constant (rational) factor, we may now arrange $c_{n}(\calS,\calS^{c}) \geq 1$ for all cuts $(\calS,\calS^{c})$ with $s \in \calS$ and $t \in \calS^{c}$. Then, we are in a position to apply the max flow min cut theorem: there exists a flow $f_{n} \colon E_{n} \to \R$ such that $\|f_{n}\| \geq 1$. Moreover, recalling \eqref{form8}, the norm of the flow $f_{n}$ is bounded from above by the capacity of the cut $(\{s\},\calV_{n} \setminus \{s\})$. Since there are only boundedly many edges in $E_{n}$ of the form $(s,x)$, $x \in \calV_{n}$, and the capacity of each one of them is $c_{n}(s,x) \sim 1$, we get
 \begin{equation}\label{flow} \|f_{n}\| \sim 1. \end{equation}

\subsection{Currents}\label{currentSection} In this section, we use the flow $f_{n}$ constructed above to find a metric current $T_{n}$ supported in a neighborhood of $B_{0}$. We define the current $T_{n}$ as follows. For all edges $e = (x,x') \in E_{n}$, let $\gamma_{e} \colon I_{e} := [0,d(x,x')] \to X$ be an isometric embedding with $\gamma_{e}(0) = x$ and $\gamma_{e}(d(x,x')) = x'$. Then $\calH^{1}(\gamma_{e}(I_{e})) = d(x,x') \sim r_{n}$. We define
\begin{displaymath} T_{n} := \sum_{e \in E_{n}} f_{n}(e)\gamma_{e\sharp}\llbracket I_{e} \rrbracket, \end{displaymath}
where $\llbracket I_{e} \rrbracket$ is the current discussed in Example \ref{intervalExample}.

First, we compute the boundary of $T_{n}$, based on the facts that the boundary operation is linear, and we already know (recall \eqref{form27}) the boundary of each term $\gamma_{e\sharp} \llbracket I_{e} \rrbracket$:
\begin{equation}\label{form23} \partial T_{n} = \sum_{e \in E_{n}} f_{n}(e) \partial(\gamma_{e\sharp} \llbracket I_{e} \rrbracket) = \sum_{(x,y) \in E_{n}} f_{n}(x,y) [\delta_{y} - \delta_{x}]. \end{equation}
To simplify the expression further, we use the flow property (F2) of $f_{n}$, which says that
\begin{displaymath} \sum_{(x,y) \in E_{n}} f_{n}(x,y) = 0, \qquad x \in \calV_{n} \setminus \{s,t\}. \end{displaymath}
It follows, using also the flow property (F1), namely $f(x,y) = -f(y,x)$, that if $x \in \calV_{n} \setminus \{s,t\}$, then the terms in \eqref{form23} containing $\delta_{x}$ cancel out:
\begin{displaymath} -\sum_{(x,y) \in E_{n}} f_{n}(x,y)\delta_{x} + \sum_{(y,x) \in E_{n}} f_{n}(y,x)\delta_{x} = -2\sum_{(x,y) \in E_{n}} f_{n}(x,y)\delta_{x} = 0. \end{displaymath}
Consequently, all that remains in \eqref{form23} are the terms containing $s$ and $t$:
\begin{equation}\label{boundary} \partial T_{n} = -2\sum_{(s,y) \in E_{n}} f_{n}(s,y)\delta_{s} - 2\sum_{(t,y) \in E_{n}} f_{n}(t,y)\delta_{t} = 2\|f_{n}\|\delta_{t} - 2\|f_{n}\|\delta_{s}. \end{equation}
In the last equation, we again used $f_{n}(t,y) = -f_{n}(y,t)$, and the little proposition stated in \eqref{normDef} that $\|f_{n}\| = f(\calS,\calS^{c})$ for any cut $(\calS,\calS^{c})$, in particular for $(\calS,\calS^{c}) = (\calV_{n} \setminus \{t\}, \{t\})$. Recalling \eqref{flow}, this yields that
\begin{equation}\label{form29} \|\partial T_{n}\|(X) \leq 4\|f_{n}\| \lesssim 1, \qquad n \in \N, \end{equation}
which in particular verifies that $T_n$ is a normal current.

Next, we estimate the measures $\|T_{n}\|$ and find a uniform upper bound for  $\|T_{n}\|(X)$. We recall from \eqref{form28} that $\|\gamma_{\sharp} \llbracket I_{e} \rrbracket \| = \calH^{1}\lfloor_{|\gamma_{e}|}$. It follows that
\begin{equation}\label{form24} \|T_{n}\| \leq \sum_{e \in E_{n}} |f_{n}(e)| \calH^{1}\lfloor_{\gamma_{e}(I_{e})} \leq \sum_{e \in E_{n}} c_{n}(e) \calH^{1}\lfloor_{\gamma_{e}(I_{e})}. \end{equation}
Recalling that $|f_{n}(e)| \leq c_{n}(e)$ (using the flow property (F3) and the fact that $c_{n}(x,x') = c_{n}(x',x)$ for all $(x,x') \in E_{n}$), we may now easily estimate $\|T_{n}\|(B)$ from above for all balls $B = B(x,r) \subset X$ with $r \geq r_{n}$. We claim that
\begin{equation}\label{form30} \|T_{n}\|(B) \lesssim \int_{10B} \frac{1}{\Theta(s,d(s,y))} + \frac{1}{\Theta(t,d(t,y))} \, d\mu(y), \qquad B = B(x,r) \subset X, \: r \geq r_{n}. \end{equation}
We start by disposing of a little technicality. Note that there are only boundedly many edges in $E_{n}$ of the form $(x,y)$ where $\min\{d(x,s),d(y,s)\} \leq 5r_{n}$. We denote these edges by $E_{n}(s)$, and we use the trivial estimate $c_{n}(e) \lesssim 1$ for all edges $e \in E_{n}(s)$. If $B$ happens to intersect $\gamma_{e}(I_{e})$ for one of the edges $e \in E_{n}(s)$, we estimate as follows:
\begin{displaymath} \|T_{n}\|(B) \lesssim \sum_{e \in E_{n} \setminus E_{n}(s)} c_{n}(e)\calH^{1}(\gamma_{e}(I_{e}) \cap B) + \sum_{e \in E_{n}(s)} r_{n}.\end{displaymath}
But if $B$ intersects $\gamma_{e}(I_{e})$ for an edge $e \in E_{n}(s)$, then $10B$ contains $B(s,r_{n})$, and the second term above is bounded from above by
the right hand side of \eqref{form30}:
\begin{displaymath}
 \sum_{e \in E_{n}(s)} r_{n} \lesssim \int_{B(s,r_{n})} \frac{d\mu(y)}{\Theta(s,d(s,y))} \leq \int_{10B} \frac{1}{\Theta(s,d(s,y))}+\frac{1}{\Theta(y,d(t,y))}\;d\mu(y),
\end{displaymath}
using in the first inequality that $d(s,y) \sim r_{n}$ for $y \in B(s,r_{n}) \setminus B(s,r_{n}/2) =: A(s,r_{n})$, and $\mu(A(s,r_{n})) \sim \mu(B(s,r_{n}))$ by the doubling hypothesis (this also requires $A(s,r_{n}) \neq \emptyset$, which easily follows from the path connectedness of $(X,d)$; see also \cite[(8.1.17)]{MR3363168}). We may dispose similarly of the situation where $B$ meets $\gamma_{e}(I_{e})$ for some $e \in E_{n}(t)$ (defined in the same way as $E_{n}(s)$). In other words, it remains to estimate the sum
\begin{equation}\label{form31} \sum_{e \in E_{n} \setminus (E_{n}(s) \cup E_{n}(t))} c_{n}(e)\calH^{1}(\gamma_{e}(I_{e}) \cap B) \lesssim \sum_{(x,y) \in E_{n}(B)} \frac{\mu(B(x,r_{n}))}{\Theta(s,d(s,x))} + \frac{\mu(B(x,r_{n}))}{\Theta(t,d(t,x))}. \end{equation}
where $E_{n}(B) := \{e \in E_{n} \setminus [E_{n}(s) \cup E_{n}(t)] : \gamma_{e}(I_{e}) \cap B \neq \emptyset\}$. We used in \eqref{form31} that
\begin{displaymath} d(t,y) \sim d(t,x) \quad \text{and} \quad \mu(B(y,r_{n})) \sim \mu(B(x,r_{n})), \qquad  (x,y) \in E_{n}(B). \end{displaymath}
To proceed, we recall again that every $x \in \calV_{n}$ only has boundedly many neighbours in $G_{n}$, and all the vertices $x \in \calV_{n}$ with at least one edge $(x,y) \in E_{n}(B)$ must lie at distance $\leq r_{n}$ from $B$, and at distance $\geq 5r_{n}$ from $\{s,t\}$. We denote the collection of such vertices by $\calV_{n}(B)$. The observations above, and the bounded overlap of the balls $B(x,r_{n})$, $x \in \calV_{n}$, allow us to continue \eqref{form31} as follows:
\begin{displaymath} \eqref{form31} \lesssim \sum_{x \in \calV_{n}(B)} \frac{\mu(B(x,r_{n}))}{\Theta(s,d(s,x))} + \frac{\mu(B(x,r_{n}))}{\Theta(t,d(t,x))} \lesssim \int_{10B} \frac{1}{\Theta(s,d(s,y))} + \frac{1}{\Theta(t,d(t,y))} \, d\mu(y). \end{displaymath}
Here we used, once again, that $d(s,y) \sim d(s,x)$ and $d(t,y) \sim d(t,x)$ for all $y \in B(x,r_{n})$, whenever $\dist(x,\{s,t\}) \geq 5r_{n}$. This concludes the proof of \eqref{form30}.

Finally, since the sets $\gamma_{e}(I_{e})$, $e \in E_{n}$, are geodesics connecting vertices in $\calV_{n}$, we infer that
\begin{displaymath} \spt T_{n} = \spt \|T_{n}\| \subset B(s,C_{0}d(s,t) + 5r_{n}) \subset 2B_{0}. \end{displaymath}
In particular, the supports of the currents $T_{n}$ are contained in the fixed compact set $\overline{2B_{0}}$, and
\begin{align}\label{form33} \|T_{n}\|(X) = \|T_{n}\|(2B_{0}) & \lesssim \int_{20B_{0}} \frac{1}{\Theta(s,d(s,y))} + \frac{1}{\Theta(t,d(t,y))} \, d\mu(y)\\
& \lesssim \sum_{2^{-j} \leq 2\diam(20B_{0})} 2^{-j} \left[ \frac{\mu(A(s,2^{-j}))}{\mu(B(s,2^{-j}))} + \frac{\mu(A(t,2^{-j}))}{\mu(B(t,2^{-j}))} \right] \lesssim \diam(B_{0}), \notag\end{align}
by \eqref{form30} and the doubling property of $\mu$. Recalling also \eqref{form29}, we are now in a position to use the compactness theorem for normal currents, Theorem \ref{compactness}: there exists a normal current $T$, supported on $\bar{B}_{0}$, such that
\begin{equation}\label{convergence} \lim_{m \to \infty} T_{n_{m}}(g,\pi) = T(g,\pi), \qquad (g,\pi) \in \Lip_{b}(X) \times \Lip(X), \end{equation}
for some subsequence $(n_{m})_{m \in \N}$. From the proof of the compactness theorem (\cite[Theorem 5.2]{MR1794185}) in the paper of Ambrosio and Kirchheim, in particular ``Step 2'', one can read that the measure $\|T\|$ is bounded from above by a certain measure which is obtained as a weak limit of the measures $\|T_{n_{m}}\|$. As a consequence, we infer that \eqref{form30} holds for the measure $\|T\|$, and for all balls $B(x,r) \subset X$. This fact, combined with the Lebesgue differentiation theorem, shows that $\|T\|$ is absolutely continuous with respect to the measure $\mu$, with Radon-Nikodym derivative bounded by
\begin{displaymath} \frac{d\|T\|}{d\mu}(x) \lesssim \frac{1}{\Theta(s,d(s,x))} + \frac{1}{\Theta(t,d(t,x))}. \end{displaymath}
So, to conclude the proof of Theorem \ref{main2}, it remains to show that $\partial T = C\delta_{t} - C\delta_{s}$ for some constant $C \sim 1$. From \eqref{convergence} we infer that
\begin{displaymath} \partial T(g) = \lim_{m \to \infty} \partial T_{n_{m}}(g) = \lim_{m \to \infty} (\|f_{n_{m}}\|g(t) - \|f_{n_{m}}\|g(s)), \qquad g \in \Lip_{b}(X). \end{displaymath}
Apply this to a bump function $g$ satisfying $g(t) = 1$ and $g(s) = 0$ to find that the numbers $\|f_{n_{m}}\| \sim 1$ converge to a limit $C = \partial T(g)$, which evidently also satisfies $C \sim 1$. Thus $\partial T(g) = Cg(t) - Cg(s)$ for all $g \in \Lip_{b}(X)$, which by definition means that $\partial T = C\delta_{t} - C\delta_{s}$. Finally, to be precise, the current mentioned in the definition of generalised pencil of curves can be obtained by dividing $T$ by $C$. The proof of Theorem \ref{main2} is complete.

\section{GPC and pencils of curves}\label{s:final}

In this section, we prove Theorem \ref{mainPCs}. As explained in the introduction, the implication
\begin{displaymath} X \text{ admits PCs} \quad \Longrightarrow \quad X \text{ satisfies the weak $1$-Poincar\'e inequality} \end{displaymath}
is due to Semmes for $Q$-regular spaces $(X,d,\mu)$, see \cite[Theorem B.15]{MR1414889}, and a proof of the general case can be found in \cite[Chapter 4]{MR1800917}. So, we concentrate on the reverse implication. By the main result of this paper, Theorem \ref{main}, it suffices to verify the following:
\begin{displaymath} X \text{ admits GPCs} \quad \Longrightarrow \quad X \text{ admits PCs}. \end{displaymath}
The proof is a straightforward application of a decomposition result \cite[Theorem 5.1]{MR2984069} of Paolini and Stepanov, which we now recall for the reader's convenience. Let $C,T$ be $1$-currents on $X$. We say that $C$ is a \emph{cycle} of $T$, if $C \leq T$ (recall Definition \ref{subcurrent}), and $\partial C = 0$. A $1$-current $T$ is \emph{acyclic}, if its only cycle is the trivial $1$-current $C = 0$.

By a \emph{curve} in $X$, we mean the image of a Lipschitz map $\theta \colon [0,1] \to X$. The space of curves in $X$ is denoted by $\Gamma$. We slightly abuse notation by using the letter $\theta$ to denote both Lipschitz mappings $[0,1] \to X$, and elements in $\Gamma$. There is a natural metric $d_{\Gamma}$ on $\Gamma$, defined in \cite[(2.1)]{MR2984069}, and Borel sets in $\Gamma$ are defined using the topology induced by $d_{\Gamma}$. Following the terminology above \cite[(4.1)]{MR2984069}, a finite positive Borel measure on $\Gamma$ is called a \emph{transport}. If $\theta \in \Gamma$ has an injective parametrisation, then $\theta$ is called an \emph{arc}.

By \cite[Theorem 5.1]{MR2984069}, a normal acyclic $1$-current $T$ on $X$ is decomposable in arcs. Combining \cite[Definition 4.4]{MR2984069} and \cite[Lemma 4.17]{MR2984069}, this means that there exists a transport $\eta$ on $\Gamma$ such that $\eta$ almost every curve in $\Gamma$ is an arc, and moreover the following equalities hold:
\begin{equation}\label{form25} \|T\| = \int_{\Gamma} \calH^{1}|_{\theta} \, d\eta(\theta), \end{equation}
and
\begin{equation}\label{form26} \eta(1) = (\partial T)^{+} \quad \text{and} \quad \eta(2) = (\partial T)^{-}. \end{equation}
If the reader checks carefully \cite[Lemma 4.17]{MR2984069}, then she will find the measure "$\mu_{[[\theta]]}$" (in our notation $\| [[\theta]]\|$) in place of "$\calH^{1}|_{\theta}$" in \eqref{form25}, but these measures coincide if $\theta$ is an arc. Indeed, in that case  \cite[Lemma 2.2]{MR3125271} shows that $\mu_{[[\theta]]} = \theta_{\sharp} (|\dot{\theta}|\mathcal{L}^1)$, where $|\dot{\theta}|$ denotes the metric derivative. Then it can be seen for instance by \cite[Theorem 4.1.6]{MR2039660} that $\theta_{\sharp} (|\dot{\theta}|\mathcal{L}^1)$ agrees with $\mathcal{H}^1$.

The measures $(\partial T)^{+}$ and $(\partial T)^{-}$ appearing in \eqref{form26} are the positive and negative parts of the measure $\partial T$, and $\eta(i)$ is the measure defined by
\begin{displaymath} \eta(i)(A) = \eta\{\theta \in \Gamma : \theta(i) \in A\}, \qquad i \in \{0,1\}. \end{displaymath}

Now, we apply the decomposition result to our concrete situation. Assume that $s,t \in X$, and $T$ is a GPC joining $s$ to $t$, as in Definition \ref{GPC}. There is no guarantee that $T$ is acyclic, but \cite[Proposition 3.8]{MR2984069} comes to our rescue: it gives a cycle $C \leq T$ such that $T' = T - C$ is acyclic. Clearly
\begin{displaymath} \partial T' = \partial T - \partial C = \partial T = \delta_{t} - \delta_{s}. \end{displaymath}
Moreover, $\|T'\| + \|C\| = \|T\|$ by \cite[(3.1)]{MR2984069}, so in particular $\|T'\| \leq \|T\|$. It follows that $T'$ is an acyclic GPC joining $s$ to $t$. With this argument in mind, we may assume that $T$ was acyclic to begin with.

Hence, a decomposition as in \eqref{form25}-\eqref{form26} exists. In particular, \eqref{form26} implies that
\begin{displaymath} \eta\{\theta \in \Gamma : \theta(1) = t\} = \eta(1)(\{t\}) = (\partial T)^{+}(\{t\}) = \delta_{t}(\{t\}) = 1, \end{displaymath}
and similarly $\eta\{\theta \in \Gamma : \theta(0) = s\} = 1$. So, there exists a set $\Gamma_{s,t} \subset \Gamma$ of arcs of full $\eta$ measure such that $\theta(0) = s$ and $\theta(1) = t$ for all $\theta \in \Gamma_{s,t}$. Moreover, if $g \colon X \to [0,\infty]$ is a Borel function, then \eqref{form25} shows that
\begin{displaymath} \int_{\Gamma_{s,t}} \int_{\theta} g \, d\calH^{1} \, d\eta(\theta) = \int g \, d\|T\| \lesssim \int_{B(s,C_0 d(s,t))} \frac{g(y)}{\Theta(s,d(s,y))} + \frac{g(y)}{\Theta(t,d(t,y))} \, d\mu(y),  \end{displaymath}
taking into account that $\spt \|T\| \subset B(s,C_{0}d(s,t))$ by (P2). The only point we are still missing from Definition \ref{PC} is the claim that the arcs in $\Gamma_{s,t}$ are quasiconvex. This need not be true to begin with, but we note that
\begin{displaymath} \|T\|(X) \lesssim d(s,t). \end{displaymath}
This can be easily seen from (P2)-(P3). Alternatively, our specific construction for $T$ gives this estimate, see \eqref{form33}. So, it follows from \eqref{form25} that half of the arcs $\theta \in \Gamma$, relative to the measure $\eta$, satisfy the quasiconvexity requirement $\calH^{1}(\theta) \lesssim d(s,t)$. The proof is now completed by restricting $\Gamma_{s,t}$ to this "good half" of the arcs.

\bibliographystyle{plain}
\bibliography{references}

\def\cprime{$'$}
\begin{thebibliography}{10}

\bibitem{MR1794185}
Luigi Ambrosio and Bernd Kirchheim.
\newblock Currents in metric spaces.
\newblock {\em Acta Math.}, 185(1):1--80, 2000.

\bibitem{MR2039660}
Luigi Ambrosio and Paolo Tilli.
\newblock {\em Topics on analysis in metric spaces}, volume~25 of {\em Oxford
  Lecture Series in Mathematics and its Applications}.
\newblock Oxford University Press, Oxford, 2004.

\bibitem{MR3644391}
Reinhard Diestel.
\newblock {\em Graph theory}, volume 173 of {\em Graduate Texts in
  Mathematics}.
\newblock Springer, Berlin, fifth edition, 2017.

\bibitem{2018arXiv180903861D}
E.~{Durand-Cartagena}, S.~{Eriksson-Bique}, R.~{Korte}, and
  N.~{Shanmugalingam}.
\newblock {Equivalence of two BV classes of functions in metric spaces, and
  existence of a Semmes family of curves under a $1$-Poincar\'e inequality}.
\newblock {\em ArXiv e-prints 1809.03861}, September 2018.

\bibitem{MR2729968}
L.~R. Ford, Jr. and D.~R. Fulkerson.
\newblock {\em Flows in networks}.
\newblock Princeton Landmarks in Mathematics. Princeton University Press,
  Princeton, NJ, 2010.
\newblock Paperback edition [of MR0159700], With a new foreword by Robert G.
  Bland and James B. Orlin.

\bibitem{MR1401074}
Piotr Haj{\l}asz.
\newblock Sobolev spaces on an arbitrary metric space.
\newblock {\em Potential Anal.}, 5(4):403--415, 1996.

\bibitem{MR1683160}
Piotr Haj{\l}asz and Pekka Koskela.
\newblock Sobolev met {P}oincar\'e.
\newblock {\em Mem. Amer. Math. Soc.}, 145(688):x+101, 2000.

\bibitem{MR1800917}
Juha Heinonen.
\newblock {\em Lectures on analysis on metric spaces}.
\newblock Universitext. Springer-Verlag, New York, 2001.

\bibitem{MR3363168}
Juha Heinonen, Pekka Koskela, Nageswari Shanmugalingam, and Jeremy~T. Tyson.
\newblock {\em Sobolev spaces on metric measure spaces}, volume~27 of {\em New
  Mathematical Monographs}.
\newblock Cambridge University Press, Cambridge, 2015.
\newblock An approach based on upper gradients.

\bibitem{MR2013501}
Stephen Keith.
\newblock Modulus and the {P}oincar\'e inequality on metric measure spaces.
\newblock {\em Math. Z.}, 245(2):255--292, 2003.

\bibitem{MR1748917}
T.~J. Laakso.
\newblock Ahlfors {$Q$}-regular spaces with arbitrary {$Q>1$} admitting weak
  {P}oincar\'e inequality.
\newblock {\em Geom. Funct. Anal.}, 10(1):111--123, 2000.

\bibitem{zbMATH01249699}
P.~{Mattila}.
\newblock {\em {Geometry of sets and measures in Euclidean spaces. Fractals and
  rectifiability. 1st paperback ed.}}
\newblock Cambridge: Cambridge University Press, 1st paperback ed. edition,
  1999.

\bibitem{MR2984069}
Emanuele Paolini and Eugene Stepanov.
\newblock Decomposition of acyclic normal currents in a metric space.
\newblock {\em J. Funct. Anal.}, 263(11):3358--3390, 2012.

\bibitem{MR3125271}
Kai Rajala and Stefan Wenger.
\newblock An upper gradient approach to weakly differentiable cochains.
\newblock {\em J. Math. Pures Appl. (9)}, 100(6):868--906, 2013.

\bibitem{MR1414889}
S.~Semmes.
\newblock Finding curves on general spaces through quantitative topology, with
  applications to {S}obolev and {P}oincar\'e inequalities.
\newblock {\em Selecta Math. (N.S.)}, 2(2):155--295, 1996.

\end{thebibliography}

\end{document}